\documentclass{amsart}
\usepackage{a4wide, graphicx, amsmath, amsfonts, amsthm, amssymb, latexsym,
multirow, enumerate, MnSymbol}   

\usepackage[colorlinks]{hyperref}
\hypersetup{linkcolor=blue, urlcolor=blue, citecolor=red}

\newtheorem{thm}{Theorem}[section]
\newtheorem{lemma}[thm]{Lemma}

\newcommand{\proofref}[1]{\noindent {\emph{Proof of Theorem} \ref{#1}.\ }}

\newcommand{\genset}[1]{\langle#1\rangle}
\newcommand{\pres}[2]{\langle\:#1\:|\:#2\:\rangle}
\newcommand{\N}{\mathbb{N}}
\newcommand{\im}{\operatorname{im}}
\newcommand{\pst}{\:|\:}

\begin{document}

\title[Disjoint unions of free monogenic semigroups]{A classification of
disjoint unions of two or three copies of the free
monogenic semigroup}

\author{N. Abu-Ghazalh, J. D. Mitchell,  Y. P\'eresse, N. Ru\v{s}kuc} 
\thanks{The first author is financially supported by The Ministry of Higher
Education in Saudi Arabia (Princess Nora Bint Abdul Rahman University in
Riyadh, Ref number RUG0003).}

\begin{abstract}
We prove that, up to isomorphism and anti-isomorphism, there are only two
semigroups which are the union of two copies of the free 
monogenic semigroup. Similarly,  there are only nine semigroups which are the
union of three copies of the free monogenic semigroup.
We provide finite presentations for each of these semigroups.
\end{abstract}

\maketitle

\section{Introduction and Preliminaries}

There are several well-known examples of structural theorems for semigroups,
which involve decomposing a semigroup into a disjoint union of subsemigroups.
For example, up to isomorphism, the Rees Theorem states that every completely
simple semigroup is a Rees matrix semigroup over a group $G$, and is thus a
disjoint union of copies of $G$, see \cite[Theorem 3.3.1]{Howie1995aa}; every
Clifford semigroup is a strong semilattice of groups and as such it is a disjoint
union of its maximal subgroups, see \cite[Theorem 4.2.1]{Howie1995aa}; every
commutative semigroup is a semilattice of archimedean semigroups,
see \cite[Theorem 2.2]{Grillet1995aa}. 

If $S$ is a semigroup which can be decomposed into a disjoint union of
subsemigroups, then it is natural to ask how the properties of the subsemigroups
influence $S$. For example, if the subsemigroups are finitely generated, then so
is $S$.  There are several further examples in the literature where such
questions are addressed: Ara\'ujo et al. \cite{Araujo2001aa} consider the finite
presentability of semigroups which are the disjoint union of finitely presented
subsemigroups; Golubov \cite{Golubov1975aa} showed that a semigroup which is the
disjoint union of residually finite subsemigroups is residually finite; in
\cite{Abu-Ghazalh2013aa} the authors proved that every semigroup which is a disjoint
union of finitely many copies of $\N$ is finitely presented; further references
are \cite{Gray2010ab, Ruskuc1999aa}.

In this paper we completely classify those semigroups which are the disjoint
union of two or three copies of the free monogenic semigroup. 

The main theorems of this paper are the following. 

\begin{thm} 
  \label{2copies} 
  Let $S$ be a semigroup. Then $S$ is a disjoint union of two copies of the free
  monogenic semigroup if and only if $S$ is isomorphic or anti-isomorphic to the
  semigroup defined by one of the following presentations: 
  \begin{enumerate}[\rm (i)] 
    \item $ \langle\ a,b\ |\ ab=ba=a^k\ \rangle$ for some $k\geq 1$; 
    \item $ \langle\ a,b\ |\ ab=a^2,\ ba=b^2\ \rangle$.
  \end{enumerate} 
\end{thm}

\begin{thm}\label{3copies}
Let $S$ be a semigroup. Then $S$ is  a disjoint union of three copies of the free monogenic semigroup if and only if  
$S$ is isomorphic or anti-isomorphic to the semigroup defined by one of the following presentations:
\begin{enumerate}[\rm (i)]
\item 
$\langle\  a,b,c\ |\  ab=a^i,ba=a^i,ac=a^j, ca=a^j,bc=a^k,cb=a^k \ \rangle$
where  $i+j=k+2$ and  $i,j,k\in \N$;

\item
$\langle\  a,b,c\ |\ ab=a^i,ba=a^i,ac=a^j,ca=a^j,bc=b^k,cb=b^k\ \rangle$
where $i+j+k-ik=2$ and  $i,j,k\in \N$;

\item
$\langle\  a,b,c\ |\ ab=a^i,ba=a^i,ac=a^i,ca=a^i,bc=c^2,cb=b^2\ \rangle$
where $i\in \N$;

\item
$\langle\  a,b,c\ |\ ab=a^i,ba=a^i,ac=c^2,ca=a^2,bc=c^i,cb=a^i\ \rangle$
where $i\in \N$;

\item
$\langle\  a,b,c\ |\ ab=a^i,ba=a^i,ac=c^2,ca=a^2,bc=c^i,cb=c^i\ \rangle$
where $i\in \N$;

\item 
$\langle\ a,b,c\ |\ ab=b^2,ba=a^2,ac=c^2,ca=a^2,bc=c^2,cb=b^2\ \rangle$;

\item 
$\langle\  a,b,c\ |\ ab=b^2,ba=a^2,ac=c^2,ca=b^2,bc=c^2,cb=a^2\ \rangle$;

\item
$\langle\  a,b,c\ |\ ab=b^2,ba=a^2,ac=c^2,ca=a^2,bc=c^2,cb=a^2\ \rangle$; 

\item
$\langle\  a,b,c\ |\ ab=b^2,ba=a^2,ac=b^i,ca=a^i,bc=a^i,cb=b^i\ \rangle$
where $i\in \N.$
\end{enumerate}
\end{thm}

We prove Theorems \ref{2copies} and \ref{3copies} in Sections \ref{sec 2} and
\ref{sec 3} respectively. 

Let $A$ be a set, and let $S$ be any semigroup. 
Then we denote by $A^+$ the \emph{free semigroup} on $A$,
which  consists of the non-empty words over $A$. 
Any mapping $\psi: A \rightarrow S$ can be extended in a unique way to a
homomorphism $\phi: A^+\rightarrow S$, and $A^+$ is determined up to
isomorphism by these properties. 
If $A$ is a generating set for $S$, then the identity mapping on $A$ induces an
epimorphism
$\pi : A^+\rightarrow S$.  The kernel $\ker(\pi)$ is a congruence on $S$; if
$R\subseteq A^+\times A^+$ generates this congruence we say that
$\langle A\pst R\rangle$ is a presentation for $S$.  We say that $S$
\emph{satisfies a relation} $(u,v)\in A^+\times A^+$ if $\pi(u)=\pi(v)$; we
write $u=v$ in this case.  Suppose we are given a set $R\subseteq A^+\times
A^+$ and two words $u,v\in A^+$.  
We write $u\equiv v$ if $u$ and $v$ are equal as elements of $A^{+}$.  
We say that the relation $u=v$ is a
\emph{consequence} of $R$ if there exist words $u\equiv w_1,w_2,\dots,
w_{k-1},w_k\equiv v$ ($k\geq 1$) such that for each $i=1,\dots,k-1$ we can
write $w_i\equiv \alpha_i u_i \beta_i$ and $w_{i+1}\equiv \alpha_i v_i\beta_i$
where $(u_i,v_i)\in R$ or $(v_i,u_i)\in R$. We say that $\langle A\pst
R\rangle$ is a presentation for $S$ if and only if $S$ satisfies all relations
from $R$, and every relation that $S$ satisfies is a consequence of $R$: see
\cite[Proposition 1.4.2]{Lallement1979aa}. If $A$ and $R$ are finite,
then $S$ is finitely presented.

Let $\rho$ be a congruence on a semigroup $S$, and let $\phi:S\rightarrow T$ be
a homomorphism such that $\rho \subseteq \ker \phi$. Then there is a
unique homomorphism $\beta: S/\rho \rightarrow T$ defined by 
$s/\rho\mapsto \phi(s)$ and such that $\im\beta= \im\phi$; 
\cite[Theorem 1.5.3]{Howie1995aa}.
Let $S$ be the semigroup defined by the presentation $\langle A\ |\ R\rangle $. If $T$ is any semigroup satisfying the relations $R$,  then $T$
is a homomorphic image of $S$.

\begin{lemma} \label{lem 3.7}
Let $A$ be a set, let $\rho$ be a congruence on $A^+$, let 
$\varphi:A \to\N\cup\{0\}$ be any mapping, and let 
$\psi :A^+ \to  \N\cup\{0\}$ be the unique homomorphism extending 
$\varphi$. If $\rho\subseteq \ker\psi$ and $a\in A$
such that $\varphi(a)\neq 0$, then $\langle a/\rho \rangle$ is an infinite
subsemigroup of $A^+/\rho$.
\end{lemma}
\begin{proof}
Since $\rho\subseteq\ker\psi$, it follows that 
$\overline{\psi}:A^+/\rho \rightarrow \N \cup \{0\}$ defined by 
$ \overline{\psi}(w/\rho)=\psi(w)$ is a
homomorphism. Homomorphisms map elements of
finite order to elements of finite order, and since $\varphi(a)\neq 0$ does
not have finite order, $a/\rho$ must have infinite order in $A^+/\rho$.
\end{proof}

Let $\genset{a}$ be the free monogenic semigroup. Then any two non-empty subsemigroups $S$ and $T$ of $\genset{a}$ 
have non-empty 
intersection, since $a^i\in S$ and $a^j\in T$ implies $a^{ij}\in S\cap T$. 
Let $S$ be a semigroup which is the disjoint union of $m\in \N$ copies of the free monogenic semigroup, and let $a_1, \ldots, a_m\in S$ be the generators of these copies. 
Suppose that $S$ is also the disjoint union of $n\in \N$ copies of the free monogenic semigroup. Then there exist 
$b_1, \ldots, b_n\in S$ such that  $\genset{b_1}, \ldots, \genset{b_n}$ are free,  disjoint, and 
$$S=\genset{a_1}\cup\cdots \cup \genset{a_m}=\genset{b_1}\cup\cdots \cup\genset{b_n}.$$
If $n>m$, say, then there exist $i, j$ such that $b_i, b_j\in \genset{a_k}$ for some $k$. But then 
$\genset{b_i}\cap \genset{b_j}\not=\emptyset$, a contradiction.
Hence a semigroup cannot be the disjoint union of $m$ and $n$ copies of the free monogenic semigroup 
when $n\not=m$.


\begin{lemma}
\label{homo-iso}
Let $S$ and $T$ be semigroups which are the disjoint union of $m\in \N$ copies
of the free monogenic semigroup, and let 
$A=\{a_1, \ldots, a_m\}$ and $B=\{b_1, \ldots, b_m\}$ be the
generators of these copies in $S$ and $T$, respectively. 
Then every homomorphism $\varphi:T\to S$ such that
$\varphi(a_i)=b_i$ for all $i$ is an isomorphism. 
\end{lemma}
\begin{proof}
Since $\varphi$ is surjective, it follows that the function 
$f:\{1,\ldots, m\}\to \{1,\ldots, m\}$ defined by $\varphi(a_i)\in \genset{b_{f(i)}}$ is a bijection. 

Suppose that there exist $x,y\in S$
such that $\varphi(x)=\varphi(y)$. Then there exist $a,b\in A$ such
that $x=a^i$ and $y=b^j$ for some $i,j\in \N$. It follows that
$\varphi(a)^i=\varphi(b)^j$, which implies that 
$\varphi(a), \varphi(b)\in \genset{c}$ for some $c\in B$. Hence $a=b$, since $f$ is a bijection, and so $x=y$. 
\end{proof}
Since the free monogenic semigroup is anti-isomorphic to itself, it follows that a semigroup $S$ is the disjoint union of $m$ copies of the free monogenic semigroups if and only if any semigroup anti-isomorphic to $S$ has this property.

\section{Two copies of the free monogenic semigroup}\label{sec 2}
In this section we prove Theorem \ref{2copies}.\vspace{\baselineskip}

\proofref{2copies}  
($\Leftarrow$) 
To prove the converse implication, it suffices to show that the semigroups mentioned in Theorem \ref{2copies} are disjoint unions of two copies of the free monogenic semigroup, since this is a property preserved by (anti-)isomorphisms.

Let $m\in \N$ be arbitrary and let $S$ be the semigroup defined by the
presentation
$\pres{a,b}{ab=ba=a^m}.$
It is clear that every element of $S$ is a power of $a$ or $b$, and so
$S=\genset{a}\cup \genset{b}$. 
Since there is no relation in the presentation that can be applied to a power
of $b$, it follows that $\genset{a}\cap \genset{b}=\emptyset$ and 
$\genset{b}$ is infinite. 
We show that $\genset{a}$ is infinite using Lemma \ref{lem 3.7}.
Let $\rho$ be the congruence on $\{a,b\}^+$ generated by the relations
$ab=a^m$ and $ba=a^m$, let 
$\varphi:\{a,b\} \rightarrow \N$ be defined by
$\varphi (a)=1$, $\varphi (b)=m-1$, and let 
$\psi :\{a,b\}^+ \rightarrow  \N$ be the unique homomorphism extending
$\varphi$. Then  
$$\psi(ab)=\psi(a)+\psi(b)=1+m-1=m=\psi(a^m)$$ and, 
similarly, $\psi(ba)=\psi(a^m)$. Hence $\rho\subseteq \ker \psi$ and so
$\genset{a}$ is infinite in $S$, by Lemma \ref{lem 3.7}. 

Let $T$ be the semigroup defined by the presentation 
$\pres{a,b}{ab=a^2,ba=b^2}$. Then as above $T=\genset{a}\cup
\genset{b}$. Any product of $a$ and $b$ equal to a power of $a$ must start
with $a$ and any product equal to a power of $b$ must start with $b$. 
Hence $\genset{a}\cap \genset{b}=\emptyset$. 
The proof that $\genset{a}$ and $\genset{b}$ are infinite follows using a
similar argument as above but where
$\varphi:\{a,b\}\to \N$ is defined by 
$\varphi(a)=1=\varphi(b)$. 
\vspace{\baselineskip}

($\Rightarrow$) 
Let $S$ be a semigroup which is the disjoint union of the free semigroups
$\genset{a}$ and $\genset{b}$. Clearly one of the following must hold: 
\begin{enumerate}
\item[(a)] $ab, ba\in \genset{a}$,
\item[(b)]$ab, ba\in \genset{b}$,
\item[(c)]$ab\in \genset{a}$ and $ba\in \genset{b}$,
\item[(d)]$ab\in \genset{b}$ and $ba\in \genset{a}$. 

\end{enumerate}
In case (b), $S$ is isomorphic to a semigroup satisfying (a) and in case (d), $S$ is anti-isomorphic to a semigroup satisfying (c). Hence we may assume without loss of generality that (a) or (c) hold.
\vspace{\baselineskip}

\noindent {\bf Case (a)} There exist $m,n\in \N$ such that $ab=a^m$ and $ba=a^n$. Hence $$a^{m+1}=a^ma=(ab)a=a(ba)=aa^n=a^{n+1}$$
and so $m=n$. So, in this case, $S$ is a homomorphic image of the semigroup $T$ 
defined by the presentation $\langle\ a,b\ |\ ab=ba=a^m\ \rangle$. It follows
from Lemma \ref{homo-iso} that $S$ is isomorphic to $T$. 
\vspace{\baselineskip}

\noindent{\bf Case (c)}
There exist $m,n\in \N$ such that $ab=a^m$ and $ba=b^n$. So, in this case,
$$a^{m+1}=a^ma=(ab)a=a(ba)=ab^n=a^{m}b^{n-1}=\cdots=a^{n(m-1)+1}$$
and so $m=n(m-1)$, which implies that $m=n=2$. In this case, it follows that 
$S$ is a homomorphic image of the semigroup
defined by the presentation $\langle\ a,b\ |\ ab=a^2, ba=b^2\ \rangle$, and so
by Lemma \ref{homo-iso}, $S$ is isomorphic to this semigroup.
\qed


\section{Disjoint unions of three copies of the free monogenic semigroup}
\label{sec 3}
Let $S$ be a semigroup which is the disjoint union of the free monogenic
semigroups $\genset{a}$, $\genset{b}$, and $\genset{c}$. We will show that $S$
is determined, in some sense, by the values of the products $ab, ba, ac, ca,
bc, cb$. 
To this end, define the \emph{type} of $S$ to be 
$(A,B,C,D,E,F)$ where $A,B,C,D,E,F\in \{a,b,c\}$ 
if $ab\in \genset{A}$, $ba\in \genset{B}$, $ac\in \genset{C}$, $ca\in
  \genset{D}$, $bc\in \genset{E}$, $cb\in \genset{F}$. 
There are $3^6=729$ different types and so, potentially, $729$ different cases to consider in the proof of Theorem \ref{3copies}. In order to bring this number down to a more manageable $9$ cases, we require the following observations and lemma.
 
  If $S$ has type $(A,B,C,D,E,F)$, then reversing the order of multiplication in $S$ defines a semigroup anti-isomorphic to $S$ of type $(B,A,D,C,F,E)$. We will say that the types $(A,B,C,D,E,F)$ and $(B,A,D,C,F,E)$ are anti-isomorphic. Similarly, by renaming the generators of $S$ according to some permutation $\sigma$ of the set $\{a,b,c\}$, we obtain a semigroup isomorphic to $S$ of type $(U,V,W,X,Y,Z)$. We will say that $(A,B,C,D,E,F)$ is isomorphic to $(U,V,W,X,Y,Z)$ via $\sigma$. For example, $(b,a,b,a,c,b)$ is isomorphic to $(b,a,c,a,c,a)$ via the permutation $(cba)$.
One step in the proof of Theorem \ref{3copies} is to show that every $S$ is isomorphic or anti-isomorphic to a semigroup of one of only $9$ types.

\begin{lemma} \label{thm 3.1} 
Let $S$ be a semigroup which is the disjoint union of the free monogenic
semigroups $\genset{a}$, $\genset{b}$, and $\genset{c}$. Then one of
$\genset{a}\cup \genset{b}$, $\genset{a}\cup \genset{c}$, or 
$\genset{b}\cup \genset{c}$ is a subsemigroup of $S$.
\end{lemma}
\begin{proof}
Seeking a contradiction suppose that none of $\genset{a}\cup \genset{b}$,
$\genset{a}\cup \genset{c}$, or
$\genset{b}\cup \genset{c}$ is a subsemigroup of $S$. 
Then  $ab$ or  $ba\in \genset{c}$, and $ac$ or $ca\in \genset{b}$, and 
$bc$ or $cb\in \genset{a}$. In each of these cases we will show that some
power of $a$, say, equals a power of $b$ or $c$, which will yield the
required contradiction. 

If $ab=c^i$ and $bc=a^j$, then 
$a^{j+1}=abc=c^{i+1}.$ 
If $ab=c^i$ and $ca=b^j$, then 
$b^{j+1}=cab=c^{i+1}$. 
If $ac=b^i$ and $cb=a^j$, then $b^{i+1}=acb=a^{j+1}$. 
The remaining cases follow by symmetry. 
\end{proof}

\vspace{\baselineskip}
\proofref{3copies}
($\Leftarrow$) 
We will show that the semigroup defined by any of the presentations in Theorem \ref{3copies}, and therefore any semigroup (anti-)isomorphic to it, is the disjoint union of three copies of the free monogenic semigroup.
It is straightforward to verify that every element of a semigroup defined by
any of the presentations is a power of $a$, $b$ or $c$. It
therefore suffices to show that $\genset{a}$, $\genset{b}$, and $\genset{c}$
are pairwise disjoint and infinite. 

As in the proof of Theorem \ref{2copies}, we show that $\genset{a}$,
$\genset{b}$, and $\genset{c}$ are infinite
by applying Lemma \ref{lem 3.7} to the respective 
congruences $\rho$ on $\{a,b,c\}^+$ generated by the relations in the
presentation of the relevant case, and the functions 
$\varphi:\{a,b,c\} \rightarrow \N$ defined by $\varphi(a)=1$ and
\begin{center}
  \begin{tabular}{c|c|c|c|c|c|c|c|c|c}
    &(i)&(ii)&(iii)&(iv)&(v)&(vi)&(vii)&(viii)&(ix)\\\hline
    $\varphi(b)$&$i-1$&$i-1$&$i-1$&$i-1$&$i-1$&1&1&1&1\\
    $\varphi(c)$&$j-1$&$j-1$&$i-1$&1&1&1&1&1&$i-1$\\
  \end{tabular}
\end{center}

To conclude this part of the proof we must show that $\genset{a}$,
$\genset{b}$, and $\genset{c}$ are pairwise disjoint. There are several cases
to consider.  In each of these cases, we let $S$ denote the semigroup defined
by the presentation in that case.\vspace{\baselineskip} 

\noindent{\bf Case (i).} 
In the semigroup defined by the presentation in case (i), 
no relation can be applied
to a power of $b$ or $c$, and so $\genset{a}$, $\genset{b}$, and $\genset{c}$
are disjoint (and the latter two are infinite). \vspace{\baselineskip} 

\noindent{\bf Cases (ii) to (vi) and (ix).}
The proofs that $\genset{a}$, $\genset{b}$, and $\genset{c}$
are pairwise disjoint in the semigroups defined by the presentations in
cases (ii) to (vi) and (ix) are similar, and so we present the proofs
simultaneously.  In each of these cases, let $T$ be the semigroup defined by
the respective multiplication table:
\begin{center}
 \begin{tabular}{c|ccc} 
(ii)&$a'$&$b'$&$c'$ \\ \hline 
$a'$&$a'$&$a'$&$a'$\\
$b'$&$a'$&$b'$&$b'$ \\ 
$c'$&$a'$&$b'$&$c'$\\ 
\end{tabular}\quad
\begin{tabular}{c|ccc} 
(iii)&$a'$&$b'$&$c'$ \\ \hline 
$a'$&$a'$&$a'$&$a'$\\
$b'$&$a'$&$b'$&$c'$ \\ 
$c'$&$a'$&$b'$&$c'$\\
\end{tabular}\quad
\begin{tabular}{c|ccc} 
(iv)&$a'$&$b'$&$c'$ \\ \hline 
$a'$&$a'$&$a'$&$c'$\\
$b'$&$a'$&$b'$&$c'$ \\ 
$c'$&$a'$&$a'$&$c'$\\ 
\end{tabular}\quad
\begin{tabular}{c|ccc} 
(v)&$a'$&$b'$&$c'$ \\ \hline 
$a'$&$a'$&$a'$&$c'$\\
$b'$&$a'$&$b'$&$c'$ \\ 
$c'$&$a'$&$c'$&$c'$\\ 
\end{tabular}\quad
\begin{tabular}{c|ccc} 
(vi)&$a'$&$b'$&$c'$ \\ \hline 
$a'$&$a'$&$b'$&$c'$\\
$b'$&$a'$&$b'$&$c'$ \\ 
$c'$&$a'$&$b'$&$c'$\\ 
\end{tabular}, \quad
\begin{tabular}{c|cccc} 
(xi)&$1$&$a'$&$b'$&$c'$ \\ \hline 
$1$&$1$&$a'$&$b'$&$c'$\\ 
$a'$&$a'$&$a'$&$b'$&$b'$ \\ 
$b'$&$b'$&$a'$&$b'$&$a'$\\ 
$c'$&$c'$&$a'$&$b'$&$1$\\ 
\end{tabular}   \\        
\end{center}
and let $\sigma:\{a,b,c\}\rightarrow T$ be defined by $\sigma(a)=a'$,
$\sigma(b)=b'$, and $\sigma(c)=c'$. If $\tau:\{a,b,c\}^+\to T$ is the unique
homomorphism extending $\sigma$, then it is routine to verify that
$\ker(\tau)$ contains the congruence $\rho$ generated by the relations in the
presentation for the corresponding case. 
Thus, in each case, the function $w/\rho\mapsto \tau(w)$ is a
homomorphism from $S=\{a,b,c\}^+/\rho$ onto $T$ (by \cite[Theorem
1.5.3]{Howie1995aa}).  Therefore  $\genset{a}$, $\genset{b}$, and $\genset{c}$
are pairwise disjoint. \vspace{\baselineskip} 

\noindent{\bf Case (vii).}
It is routine to verify that by applying any relation in presentation
(vii) to $w\in \{a,b,c\}^{+}$ ending in $a^2$, $cb$, or $ba$ we obtain another
word ending $a^2$, $cb$, or $ba$. Hence $\genset{a}$ is disjoint from
$\genset{b}\cup \genset{c}$. 
Similarly, by considering
words ending $ab$, $ca$, or $b^2$, it can be shown that
$\genset{b}$ is disjoint from $\genset{a}\cup \genset{c}$.
Hence $\genset{a}$, $\genset{b}$, and $\genset{c}$ are pairwise disjoint.
\vspace{\baselineskip} 

\noindent{\bf Case (viii).}
As in the previous case, it can be shown that applying any relation from the
presentation in (viii) to $w\in \{a,b,c\}^{+}$ ending $c$ we obtain another
such word. Hence $\genset{c}$ is disjoint from $\genset{a}\cup\genset{b}$.
Similarly, by considering words ending $ab$ or $b^2$, it can be shown that
$\genset{b}$ is disjoint from $\genset{a}\cup \genset{c}$, as required. 
\vspace{\baselineskip} 

($\Rightarrow$)        
By Lemma \ref{thm 3.1}, we may assume without loss of generality that 
$\genset{a}\cup \genset{b} $ is a subsemigroup of $S$.
Furthermore, we may assume that $S$ has type $(a,a,*,*,*,*)$ or $(b,a,*,*,*,*)$, since type $(b,b,*,*,*,*)$ is isomorphic to $(a,a,*,*,*,*)$ via $(ab)$ and $(a,b,*,*,*,*)$ is anti-isomorphic to $(b,a,*,*,*,*)$.

We will show that,
up to isomorphism and anti-isomorphism, the only possible types for $S$ are:
$(a,a,a,a,a,a)$, $(a,a,a,a,b,b)$, $(a,a,a,a,c,b)$, $(a,a,c,a,c,a)$, 
$(a,a,c,a,c,c)$, $(b,a,c,a,c,b)$, $(b,a,c,b,c,a)$, $(b,a,c,a,c,a)$, 
$(b,a,b,a,a,b)$.  

Suppose $S$ has type $(a,a,*,*,*,*)$.  By Theorem \ref{2copies}, it follows that $ab=a^i=ba$ for some $i\in \N$. 
If $ac=b^j$ for some $j\in\N$, then 
$b(ac)=b^{j+1}$ and $(ba)c=a^ic$. If $i=1$, then $(ba)c=b^{j}$ and so $j=j+1$,
a contradiction. If $i>1$ them 
$(ba)c=a^ic=a^{i-1}b^j\in \genset{a}$, which is also a contradiction. 
By symmetry, we obtain a contradiction under the assumption that 
$ca\in \genset{b}$. 
It follows that $ac,ca\in \genset{a}\cup \genset{c}$. 

In the case that $ac,ca\in \genset{a}$, we
will show that $bc, cb\in \genset{a}$ or $bc,cb\in \genset{b}\cup \genset{c}$. 
Suppose not. Then, say, $bc\in \genset{a}$ and $cb\in \genset{b}\cup
\genset{c}$. It follows that $(cb)(cb)\in \genset{b}\cup \genset{c}$ but
$(c(bc))b\in \genset{a}$. The case that $cb\in \genset{a}$ and $bc\in
\genset{b}\cup \genset{c}$ follows by a similar argument. 
We have shown that the only $S$ of type $(a,a,a,a,*,*)$ are 
$(a,a,a,a,a,a)$, $(a,a,a,a,b,b)$, $(a,a,a,a,b,c)$, $(a,a,a,a,c,b)$, and 
$(a,a,a,a,c,c)$.

If $ac\in \genset{c}$ and $bc\in \genset{a}\cup \genset{b}$, then $a(bc)\in
\genset{a}$ and $(ab)c\in \genset{c}$, a contradiction. Similarly, if
$ca\in \genset{c}$ and $cb\in \genset{a}\cup \genset{b}$, then $(cb)a\in
\genset{a}$ but $c(ba)\in
\genset{c}$. It follows that if $S$ is of type $(a,a,c,*,*,*)$, then 
$S$ is of type $(a,a,c,*,c,*)$ and if $S$ is of type $(a,a,*,c,*,*)$, then $S$
is of type $(a,a,*,c,*,c)$. It follows that the only $S$ of type
$(a,a,c,c,*,*)$ are of type $(a,a,c,c,c,c)$. 

If $ca\in \genset{c}$ and $bc\in \genset{b}$, then $b(ca)\in \genset{b}$ and $(bc)a\in \genset{a}$, a
contradiction. 
Hence the only $S$ of type $(a, a, a, c, *, c)$ are of type
$(a,a,a,c,a,c)$ or $(a,a,a,c,c,c)$. It follows by symmetry that the
only $S$ of type $(a, a, c, a, c, *)$ are of type
$(a,a,c,a,c,a)$ or $(a,a,c,a,c,c)$.

Therefore if $S$ is a semigroup of type $(a,a,*,*,*,*)$, then
$S$ has one of the following types: $(a,a,a,a,a,a)$,
$(a,a,a,a,b,b)$, $(a,a,a,a,c,c)$, $(a,a,a,a,b,c)$,
$(a,a,a,a,c,b)$, $(a,a,c,c,c,c)$, $(a,a,c,a,c,a)$, $(a,a,c,a,c,c)$, 
$(a,a,a,c,a,c)$, $(a,a,a,c,c,c)$. 

However, $(a,a,a,a,b,c)$, $(a,a,a,c,a,c)$ and $(a,a,a,c,c,c)$ are anti-isomorphic to $(a,a,a,a,c,b)$, $(a,a,c,a,c,a)$ and $(a,a,c,a,c,c)$, respectively.
Moreover, $(a,a,a,a,c,c)$ is isomorphic to $(a,a,a,a,b,b)$ via $(bc)$ and $(a,a,c,c,c,c)$ is isomorphic to $(a,a,a,a,b,b)$ via $(cab)$. This
leaves the tuples given at the start of this part of the proof. 

Suppose that $S$ has a type $(b,a,*,*,*,*)$ which is not isomorphic to $(a,a,*,*,*,*)$. Then $S$ does not have type $(*,*,a,a,*,*)$, $(*,*,c,c,*,*)$, $(*,*,*,*,b,b)$ or $(*,*,*,*,c,c)$. We prove that the only possible $S$ of type
$(b,a,*,*,*,*)$ are:
$$(b,a,b,a,a,b), (b,a,b,a,c,b), (b,a,c,b,c,a), (b,a,c,b,c,b), (b,a,c,a,c,a),
(b,a,c,a,a,b), (b,a,c,a,c,b).$$

Since $(b,a,b,a,c,b)$, $(b,a,c,b,c,b)$ and $(b,a,c,a,a,b)$ are isomorphic to $(b,a,c,a,c,a)$ via $(cba)$, $(ab)$ and $(bc)$, respectively, this 
will leave the tuples given at the start of this part of the proof. 
It suffices to prove the following:

\begin{enumerate}[(a)]
\item if $S$ has type $(b,a,a,*,*,*)$, then it has already been considered;
\item if $S$ has type $(b,a,b,*,*,*)$, then it has type
$(b,a,b,a,*,b)$;
\item if $S$ has type $(b,a,c,b,*,*)$, then it has type $(b,a,c,b,c,*)$;
\item if $S$ has type $(b,a,c,a,*,a)$, then it has type $(b,a,c,a,c,a)$;
\item $S$ cannot have type $(b,a,c,a,*,c)$.
\end{enumerate}
\vspace{\baselineskip}

\noindent{\bf Case (a).}
If $cb\in \genset{a}\cup \genset{c}$, then
$a(cb)\in \genset{a}$ but $(ac)b\in \genset{b}$, a contradiction. 
Hence $cb\in \genset{b}$. 
If $ca\in \genset{c}$, then $(cb)a\in \genset{a}$ but 
$c(ba)\in \genset{c}$. If $ca\in \genset{b}$, then
$a(ca)\in \genset{b}$ but $(ac)a\in \genset{a}$. Hence $ca\in\genset{a}$ and so
$S$ has type $(b,a,a,a,*,*)$. It follows that $S$ is isomorphic or
anti-isomorphic to a semigroup of type $(a,a,*,*,*,*)$. 
\vspace{\baselineskip}

\noindent{\bf Case (b).}
If $ca\in \genset{b}\cup\genset{c}$, then $a(ca)\in \genset{b}$ but $(ac)a\in
\genset{a}$. Hence $ca\in \genset{a}$ and $S$ has type $(b,a,b,a,*,*)$. 
If $cb\in\genset{c}$, then $c(ab)\in \genset{c}$ but $(ca)b\in \genset{b}$. 
If $cb\in\genset{a}$, then $a(cb)\in \genset{a}$ and $(ac)b\in\genset{b}$.
Thus $S$ has type $(b,a,b,a,*,b)$, as required. 
\vspace{\baselineskip}

\noindent{\bf Case (c).}
If $bc\in \genset{a}$, then $b(ca)\in \genset{b}$ but $(bc)a\in \genset{a}$. 
If $bc\in \genset{b}$, then $b(ac)\in \genset{b}$ but $(ba)c\in \genset{c}$. 
Hence $S$ has type $(b,a,c,b,c,*)$, as required. 
\vspace{\baselineskip}

\noindent{\bf Case (d).}
If $bc\in \genset{a}\cup\genset{b}$, then $b(cb)\in \genset{a}$ but $(bc)b\in
\genset{b}$. Therefore $S$ has type $(b,a,c,a,c,a)$, as required.
\vspace{\baselineskip}

\noindent{\bf Case (e).}
If $S$ has type $(b,a,c,a,*,c)$, then $c(ab)\in \genset{c}$ but
$(ca)b\in\genset{b}$, a contradiction.
\vspace{\baselineskip}

It remains to show that if $S$ has one of the types
given at the start of the proof, then $S$ is isomorphic to a semigroup defined
by one of the presentations in the theorem. By Lemma
\ref{homo-iso}, it suffices to show that the generators $a$, $b$, and $c$ of
$S$ satisfy the relations in one of the presentations. 

Suppose that $S$ has type $(a,a,a,a,a,a)$.
Then $\genset{a}\cup\genset{b}$ and $\genset{a}\cup\genset{c}$ are
subsemigroups of $S$ and hence by Theorem \ref{2copies} $ab=a^i=ba$ and
$ac=a^j=ca$ for some $i,j\in\N$. If $bc=a^k$ and $cb=a^l$ in $S$, then 
$a^{k-1+i}=a^kb=bcb=ba^l=a^{l-1+i}$ and so $k=l$. 
Also $a^{i+j}=a(bc)a=a^{k+2}$ and so $i+j=k+2$. 
Thus the presentation in (i)
defines a semigroup isomorphic to $S$.

If $S$ has type $(a,a,a,a,b,b)$, then $\genset{a}\cup\genset{b}$,
$\genset{a}\cup\genset{c}$, and $\genset{b}\cup\genset{c}$ are subsemigroups of
$S$, and so $ab=ba=a^i$, $ac=ca=a^j$, and $bc=cb=b^k$ for some $i,j,k\in \N$ (by
Theorem \ref{2copies}). Also $a^{i+j}=abca=ab^ka=a^{ik-k+2}$, and so
$i+j=ik-k+2$. 
So, the presentation in (ii) defines $S$.

If $S$ has type $(a,a,a,a,c,b)$, then again $\genset{a}\cup\genset{b}$,
$\genset{a}\cup\genset{c}$, and $\genset{b}\cup\genset{c}$ are subsemigroups of
$S$. Hence, by Theorem \ref{2copies}, $ab=ba=a^i$, $ac=ca=a^j$, $bc=c^2$, and
$cb=b^2$. Also $a^{2i-1}=a^ib=ab^2=acb=a^jb=a^{j-1+i}$ and so $i=j$, 
and $S$ is defined by the presentation in (iii).

If $S$ has type $(a,a,c,a,c,a)$, then $\genset{a}\cup\genset{b}$ and 
$\genset{a}\cup\genset{c}$ are subsemigroups of $S$. Hence $ab=ba=a^i$ for some
$i\in \N$, $ac=c^2$, and $ca=a^2$. If $bc=c^{j}$ and $cb=a^k$ for some
$j,k\in\N$, then $c^{j+2}=acbc=a^{k+1}c=c^{k+2}$ and so $j=k$. Furthermore, 
$a^{i+k}=abcb=ac^jb=c^{j+1}b=c^ja^k=a^{j+k}$ and so $i=j$, 
and $S$ is defined by the presentation in (iv).

If $S$ has type $(a,a,c,a,c,c)$, then $\genset{a}\cup\genset{b}$, 
$\genset{a}\cup\genset{c}$, and $\genset{b}\cup\genset{c}$ are subsemigroups of
$S$. Hence $ab=ba=a^i$, $ac=c^2$, $ca=a^2$, and $bc=cb=c^j$ for some
$i,j\in\N$. It follows that $a^{i+2}=abca=ac^ja=a^{j+2}$ and so $i=j$.  
This implies that $S$ is defined by the presentation in (v).

If $S$ has type $(b,a,c,a,c,b)$, 
then $\genset{a}\cup\genset{b}$, $\genset{a}\cup\genset{c}$, and
$\genset{b}\cup\genset{c}$ are subsemigroups of $S$ and so, by Theorem
\ref{2copies}, $S$ is defined by the presentation in
(vi).

If $S$ has type $(b,a,c,b,c,a)$,
then $\genset{a}\cup\genset{b}$ is a subsemigroup of $S$ and so $ab=b^2$ and
$ba=a^2$ by Theorem \ref{2copies}. Suppose that $ac=c^i$, $ca=b^j$, $bc=c^k$,
and $cb=a^l$.
Then $c^{2k-1}=bc^k=b^2c=(ab)c=a(bc)=ac^k=c^{i+k-1}$
which implies that $2k=i+k$ and so $i=k$.
Also
$c^{il-l+1}=a^{l-1}c^i=a^lc=(cb)c=c(bc)=cc^k=c^{k+1}$
and so $k=il-l$. Thus, since $i=k$, it follows that $k=(k-1)l$ and so $i=k=l=2$.
Finally, $a^{l+1}=(cb)a=c(ba)=ca^2=b^ja=a^{j+1}$ and so $j=l=2$.
We have shown that $S$ is defined by the presentation in (vii). 

If $S$ has type $(b,a,c,a,c,a)$, then $\genset{a}\cup\genset{b}$ and
$\genset{a}\cup \genset{c}$ are subsemigroups of $S$ and so $ab=b^2$, $ba=a^2$,
$ac=c^2$, and $ca=a^2$ by Theorem \ref{2copies}.  
If $bc=c^k$ and $cb=a^l$, then 
$a^{l+1}=ba^l=b(cb)=(bc)b=c^kb=c^{k-1}a^l=a^{k+l-1}$ and so $k=2$. 
Also $c^3=c(bc)=(cb)c=a^lc=c^{l+1}$ which implies that $l=2$. It follows that
$S$ is defined by the presentation in (viii).  

If $S$ has type $(b,a,b,a,a,b)$, then $\genset{a}\cup \genset{b}$ is a
subsemigroup of $S$ and so $ab=b^2$ and $ba=a^2$ by Theorem \ref{2copies}. 
Suppose $ac=b^i$, $ca=a^j$, $bc=a^k$, and $cb=b^l$. Then 
$a^{j+2}=baca=b^{i+1}a=a^{i+2}$ and so $i=j$. Also 
$b^{l+2}=abcb=a^{k+1}b=b^{k+2}$ and so $k=l$. Finally,
$b^{k+2}=a^kb^2=a^kab=b(ca)b=ba^jb=a^{j+1}b=b^{j+2}$ and so $j=k$. 
It follows that $S$ is defined by the presentation in (ix), and the proof is
complete. 
\qed

\bibliographystyle{plain}

\vspace{\baselineskip}

\begin{flushleft}
  School of Mathematics and Statistics\\
  University of St Andrews\\
  St Andrews KY16 9SS\\
  Scotland, U.K.\\
\smallskip
\texttt{\{nabilah,jamesm,yperesse,nik\}@mcs.st-and.ac.uk}
\end{flushleft}

\end{document}